\documentclass[pdflatex,sn-mathphys-num]{sn-jnl}

\usepackage{subcaption}
\usepackage[scr=boondox,cal=esstix]{mathalpha}

\usepackage{graphicx}%
\usepackage{multirow}%
\usepackage{amsmath,amssymb,amsfonts}%
\usepackage{amsthm}%
\usepackage{mathrsfs}%
\usepackage[title]{appendix}%
\usepackage{xcolor}%
\usepackage{textcomp}%
\usepackage{manyfoot}%
\usepackage{booktabs}%
\usepackage{algorithm}%
\usepackage{algorithmicx}%
\usepackage{algpseudocode}%
\usepackage{listings}%

\theoremstyle{thmstyleone}%
\newtheorem{theorem}{Theorem}
\newtheorem{lemma}[theorem]{Lemma}%

\theoremstyle{thmstyletwo}%
\newtheorem{remark}{Remark}%

\theoremstyle{thmstylethree}%
\newtheorem{definition}{Definition}%

\newcommand{\ra}{\rightarrow}
\newcommand{\Ra}{\Rightarrow}

\newcommand{\re}{\text{Re}}

  \def\BA{{\bf A}}

\def\Bb{{\bf b}}  \def\BB{{\bf B}}

\def\Bc{{\bf c}}  \def\BC{{\bf C}}
 \def\bbC{{\mathbb{C}}}

\def\Be{{\bf e}}

\def\CG{{\cal G}}
\def\Cg{{\cal g}}
\def\Bh{{\bf h}}  \def\BH{{\bf H}}

  \def\BI{{\bf I}}
 \def\bbI{{\mathbb{I}}}

\def\Bp{{\bf p}}  \def\BP{{\bf P}}
 
\def\CP{{\cal P}}
\def\Bq{{\bf q}}  \def\BQ{{\bf Q}}

  \def\BR{{\bf R}}
 \def\bbR{{\mathbb{R}}}

\def\Bu{{\bf u}}

\def\Bv{{\bf v}}  \def\BV{{\bf V}}

\def\Bw{{\bf w}}  \def\BW{{\bf W}}

\def\Bx{{\bf x}}  \def\BX{{\bf X}}

\def\By{{\bf y}}  \def\BY{{\bf Y}}

 \def\bbZ{{\mathbb{Z}}}

\def\BLambda{{\boldsymbol{\Lambda}}}
\def\BAs{\BA{\kern-1pt}}      
\def\BPs{\BP{\kern-1.2pt}}    
\def\BVs{\BV{\kern-1.2pt}}    
\def\CPs{\CP{\kern-1.2pt}}    

\def\ii{\dot{\imath\!\imath}}
\def\Lop{\mathscr{L}}

\def\Balpha{{\boldsymbol{\alpha}}}
\def\BLambda{{\boldsymbol{\Lambda}}}
\def\Bpsi{{\boldsymbol{\psi}}}
\def\BPsi{{\boldsymbol{\Psi}}}

\def\BLambda{{\boldsymbol{\Lambda}}}

\def\Bhat{\widehat{\mathbf{B}}}
\def\bhat{\widehat{\mathbf{b}}}
\def\Chat{\widehat{\mathbf{C}}}
\def\chat{\widehat{\mathbf{c}}}

\def\Atil{\widetilde{\mathbf{A}}}
\def\Etil{\widetilde{\mathbf{E}}}
\def\btil{\widetilde{\mathbf{b}}}

\def\Ctil{\widetilde{\mathbf{C}}}

\def\htil{\widetilde{\mathbf{h}}}

\DeclareFontFamily{U}{mathx}{}
\DeclareFontShape{U}{mathx}{m}{n}{<-> mathx10}{}
\DeclareSymbolFont{mathx}{U}{mathx}{m}{n}
\DeclareMathAccent{\widecheck}{0}{mathx}{"71}

\begin{document}

\title[Partial Floquet and MOR of LTP Systems]{ Partial Floquet Transformation and  
Model Order Reduction of Linear  
Time-Periodic Systems
}


\author*[1]{\fnm{Sam} \sur{Bender}} \email{skbender@vt.edu}

\author[1]{\fnm{Christopher} \sur{Beattie}} \email{beattie@vt.edu}

\affil[1]{\orgdiv{Department of Mathematics}, \orgname{Virginia Polytechnic Institute\\ and State University}, \orgaddress{\city{Blacksburg},  \state{Virginia} \postcode{24061}, \country{USA}}}


\abstract{Time-periodic dynamical systems occur commonly both in nature and as engineered systems.   Large-scale linear time-periodic dynamical systems, for example, may arise through linearization of a nonlinear system about a given periodic solution (possibly as a consequence of a baseline periodic forcing) with subsequent spatial discretization. The potential need to simulate responses to a wide variety of input profiles (viewed as perturbations off a baseline periodic forcing) creates a potent incentive for effective model reduction strategies applicable to linear time-periodic (\textsc{ltp}) systems.  Classical approaches that take into account the underlying time-periodic system structure often utilize the Floquet transform; however, computation of the Floquet transform is typically intractable for large order systems.  In this paper, we develop the notion of a \emph{partial Floquet transformation} connected to selected invariant subspaces of a time-varying differential operator associated with the \textsc{ltp} system.  We modify and repurpose the \emph{Dominant Pole Algorithm} of Rommes to identify effective invariant subspaces useful for model reduction.  We discuss the construction of associated partial Floquet transformations and time-varying reduction bases with which to produce effective reduced-order \textsc{ltp} models and illustrate the process on a simple time-periodic system. }



\maketitle

\section{Introduction}\label{intro}

Time-periodic dynamical systems occur commonly, both in nature and as engineered systems, often as a consequence of periodic forcing due to rotation (e.g., the Earth's rotation generates both tidal gravity forces and diurnal temperature gradients that cyclically drive atmospheric and ocean flows;  gyroscopic forces can generate periodic forcing causing significant vibration and noise in vehicles).   
More broadly, periodic phenomena can occur through the emergence of a dynamic balance between inertial and various restoring forces.  For example, a structure exposed to an otherwise steady wind- or current-flow can experience large oscillations caused by vortex shedding or flutter.  This can be destructive (the Tacoma Narrows bridge failure is a famous example); but there can be positive effects as well (e.g., high-efficiency wind turbines may take advantage of these effects).  \emph{Linear} time-periodic (\textsc{ltp}) systems play a fundamental role in the analysis, simulation, and control of such phenomena even when the underlying models reflect fundamentally nonlinear dynamics, since by their character the periodic phenomena of interest emerge as components of a ``center manifold" and must themselves be stable at least when subjected to small perturbations.  Were this not the case, say for a natural system, oscillatory phenomena would not generally be observed, while for an engineered system, they would not generally be desired.   Beyond this, computational strategies for extracting periodic solutions of nonlinear systems necessitate repeated solution of linear(ized) periodic systems, and this leads (naturally) to the the question of effective model order reduction strategies for such systems. See, e.g., \cite{bittantiZeroDynamicsHelicopter1996,gazzolaTacoma2022,OutOfControl200}.

In Section 2, we introduce the analytic setting that provides the framework for our work. We also define the \emph{partial Floquet transformation} and connect it with the classical Floquet transformation.  The framework developed in Section 2 identifies invariant subspaces of a time-varying differential operator associated to the \textsc{ltp} system dynamics with families of reduction bases and hence to classes of \textsc{ltp}-reduced-order models.  In Section 3, we consider the problem of identifying ``good" reduced-order models by identifying in turn ``good" invariant subspaces.  This is done by developing notions of \emph{pole dominance} in an \textsc{ltp} setting.  In Section 4, we then modify and repurpose the \emph{Dominant Pole Algorithm} of Rommes \cite{rommesDPA} to account for our notion of \textsc{ltp}-\emph{pole dominance} and use it to identify effective \textsc{ltp}-reduced-order models. We illustrate its use on a simple example in Section 5 and provide concluding remarks in Section 6.

\section{Problem Setting}\label{sec2_problemSetting}

We consider time-periodic dynamical systems of the form:
\begin{equation}
\label{eq:FOM}
    G: \left\{
	\begin{aligned}
			\dot{\mathbf{x}}(t) &= \mathbf{A}(t) \mathbf{x}(t) + \mathbf{b}(t) u(t)\\
		y(t) &= \Bc(t)^* \mathbf{x}(t)
	\end{aligned},\right.
\end{equation}
where for each $t\in\mathbb{R}$, $\mathbf{A}(t) \in \mathbb{C}^{n \times n}$ and $\mathbf{b}(t), \mathbf{c}(t)\in\mathbb{C}^{n}$. We further assume that $\mathbf{A}(t)$, $\mathbf{b}(t)$, and $\mathbf{c}(t)$ are locally integrable and $T$-periodic (i.e., $\mathbf{A}(t) = \mathbf{A}(t+T)$, $\mathbf{b}(t) = \mathbf{b}(t+T)$ and $\mathbf{c}(t) = \mathbf{c}(t+T)$ $t$-almost everywhere for a fixed $T>0$), and so, the Carath\'{e}odory conditions for existence and uniqueness of  solutions $\mathbf{x}(t)$ hold. 
 
 For any given order $r\ll n$, we seek a \emph{reduced-order time-periodic linear model},
		\begin{equation}
		\label{eq:ROM}
         \begin{aligned}
			\dot{\mathbf{x}}_r(t) &= \mathbf{A}_r(t)\mathbf{x}_r(t) + \mathbf{b}_r(t){u}(t),\\[2mm]
			{y}_r(t) &= \mathbf{c}_r(t)^*\mathbf{x}_r(t)
		\end{aligned} 
	\end{equation}
where  $\mathbf{A}_r(t) \in \mathbb{C}^{r \times r}$ and $\mathbf{b}_r(t),\, \mathbf{c}_r(t) \in \mathbb{C}^{r}$  are obtained via projection so as to be also locally integrable and $T$-periodic, and furthermore defined in such a way so that $ {y}_r(t) \approx {y}(t)$ over a wide class of admissible inputs ${u}(t)$.	

\subsection{The Floquet Transform}
\label{subsec:FloqTrans}
Given an \textsc{ltp}  system as in \eqref{eq:FOM}, the \emph{Floquet transformation} is defined via the \emph{monodromy matrix}, i.e., the fundamental solution matrix for the homogeneous system evaluated at $T$ (the system period); see e.g., \cite{chiconeODE}. The fundamental solution matrix is an absolutely continuous matrix-valued function $\Phi(t)\in\mathbb{C}^{n\times n}$ such that $\Phi(0)=\mathbf{I}$ and 
$\dot{\Phi} = \mathbf{A}(t) \Phi$; the monodromy matrix is then $\Phi(T)$. With a suitable choice of branch cut for the (complex) logarithm, one may define $\mathbf{R}=\frac{1}{T}\log \Phi(T)$ from which it follows that  $\mathbf{P}(t)= \Phi(t)\exp(-t\mathbf{R})$ 
is $T$-periodic. This leads to 
a time-periodic change of variable $\mathbf{z}(t) = \mathbf{P}(t)^{-1} \mathbf{x}(t)$ such that 
\begin{equation} \label{CBTpostFloquet}
	\begin{aligned}
			 \dot{\mathbf{x}}(t) &= \mathbf{A}(t) \mathbf{x}(t) + \mathbf{b}(t) {u}(t)\\
		{y}(t) &= \mathbf{c}(t)^* \mathbf{x}(t)
	\end{aligned} ~ \Longrightarrow ~
	\left\{\begin{aligned}
		\dot{\mathbf{z}}(t) &= \mathbf{R} \mathbf{z}(t) + \mathbf{P}(t)^{-1}\mathbf{b}(t) {u}(t)\\
		{y}(t) &= \mathbf{c}(t)^* \mathbf{P}(t) \mathbf{z}(t).
	\end{aligned} \right.
          \end{equation}
A key consequence of this transformation is that the time-dependence in the system has now been isolated in the input/output ports; we term this system structure a \emph{port-isolated \textsc{ltp} system}.  This transformation illuminates conditions for stability for \eqref{eq:FOM}: the original \textsc{ltp} system is \emph{asymptotically stable} (and hence \textsc{bibo} stable, see e.g., \cite[\S30]{brockett2015FinDimLinSys}) if and only if $\mathbf{R}$ in \eqref{CBTpostFloquet} is a stable matrix. Our interest in this transformation is that it allows one to take advantage of powerful model reduction methods that originally were designed for linear time-invariant systems.

Although Floquet transformation is not normally viewed as a computational tool, for problems of small to moderate size, algorithms have been developed over the past decade that make effective use of Fourier spectral methods to identify $\mathbf{R}$ and $\mathbf{P}(t)$ in a numerically stable way (See e.g.,\cite{deconinckComputingSpectraLinear2006,mooreFloquetTheoryComputational2005}).  
These approaches presume access to the monodromy matrix, which makes them intractable for large-scale \textsc{ltp}  systems.  
Our approach builds up truncated Floquet bases (columns of $\mathbf{P}(t)$), which will serve for the construction of reduced order models.

Floquet transformations are useful yet costly to determine for large-scale \textsc{ltp}  systems.  
The basic transformation follows by noting that the principal fundamental solution matrix, $\Phi(t)$, 
for the homogenous \textsc{ltp} system, $ \dot{\bf x}(t) = {\bf A}(t) {\bf x}(t)$, has special structure:
$\Phi(t) = {\bf P}(t) e^{t\,{\bf R}} $
where ${\bf P}(t) = {\bf P}(t+T)$ is nonsingular and periodic for all $t$ and ${\bf R}$ is a constant
matrix.
As one may note in (\ref{CBTpostFloquet}), the (time-dependent) basis defined in the columns of ${\bf P}(t)$ 
can be used to transform a nonhomogeneous \textsc{ltp} system into a corresponding time-invariant system that may be more amenable to 
analysis (and reduction).   However, the computational bottleneck one faces in direct implementation lies in evaluation of the principal fundamental solution matrix over a period, $\Phi(t)$, followed by (explicit) computation of the matrix logarithm of the monodromy matrix, $\Phi(T)$.  Evaluation of the monodromy matrix becomes computationally intractable as state space dimension increases, though this approach can be found in engineering practice (e.g., \cite{wangHale2001OnMonodromyComp}); computation of the matrix logarithm is numerically delicate but recent advances have made computation of matrix functions less daunting (e.g., \cite{highamAlMohy2010compMatrixFunc,HaleHighamTref2008CompfAbyContInt}).   Nonetheless, difficulties persist for large dimension. 

Recent numerically stable approaches to Floquet transformation are built on Fourier spectral approximation for $\mathbf{P}(t)$ \cite{mooreFloquetTheoryComputational2005}, making the transformation computationally tractable for modest order. One computational innovation for these approaches lies in reformulating the constraint defining $\mathbf{R}$ and $\mathbf{P}(t)$ as an equivalent boundary value problem, 
\begin{equation} \label{FloqBVP}
\dot{\mathbf{P}}(t) = \mathbf{A}(t)\mathbf{P}(t) -\mathbf{P}(t)\mathbf{R},
\end{equation}
where $\mathbf{R}$ is constant and $\mathbf{P}(0)=\mathbf{P}(T)=\mathbf{I}$.  The columns of $\mathbf{P}(t)$ span an 
$n$-dimensional invariant subspace of the linear map $\mathscr{L}=-\frac{d}{dt}+\mathbf{A}(t)$.  Indeed, if 
$(\rho,\mathbf{v})$ is an eigenpair for the matrix $\mathbf{R}$, then $\mathscr{L}(\mathbf{p})=\rho \, \mathbf{p}$ for $\mathbf{p}(t)= \mathbf{P}(t) \mathbf{v}$.   
Define the 
(complex) Hilbert space, $\mathscr{H}$, of vector-valued functions on $[0,T]$ having square integrable components, equipped with an inner product: $\langle \Bw, \Bv \rangle =\frac1T \int_0^T \Bw(t)^*\Bv(t)\,dt$. 
We define $\mathsf{Dom}(\mathscr{L})$ as those $\Bv\in\mathscr{H}$ that have absolutely continuous, $T$-periodic components (thus being differentiable almost everywhere in $[0,T]$), such that the derivative $\dot{\Bv}\in\mathscr{H}$, as well.  In this setting, $\mathscr{L}$ is a densely defined spectral operator on $\mathscr{H}$.  Although the elements $\mathbf{v}\in\mathscr{H}$ are vector-valued functions on $[0,T]$, each may be extended unambiguously to a periodic vector-valued function $\mathbf{v}(t)$ defined for all $t\in\mathbb{R}$. We denote this periodic extension with no change in notation. 

Observe that for the system base frequency ${\omega}=\frac{2\pi}{T}$ and any integer $k$, \eqref{FloqBVP} implies 
\begin{align} 
\left(\mathbf{P}(t)e^{-\dot{\imath\!\imath}{\omega}kt}\right)^{\bullet}&=\dot{\mathbf{P}}(t)e^{-\dot{\imath\!\imath}{\omega}kt}-\dot{\imath\!\imath}{\omega}k\mathbf{P}(t)e^{-\dot{\imath\!\imath}{\omega}kt} \nonumber \\
&=\mathbf{A}(t)\mathbf{P}(t)e^{-\dot{\imath\!\imath}{\omega}kt} -\left(\mathbf{P}(t)e^{-\dot{\imath\!\imath}{\omega}kt}\right)\left(\mathbf{R}+\dot{\imath\!\imath}{\omega}k\mathbf{I}\right) \label{FloqBVPtransl}
\end{align}
and $\left(\mathbf{P}(t)e^{-\dot{\imath\!\imath}{\omega}kt}\right)$ is again $T$-periodic since 
$$
\mathbf{P}(t+T)e^{-\dot{\imath\!\imath}{\omega}k(t+T)}=\mathbf{P}(t)e^{-\dot{\imath\!\imath}{\omega}kt}e^{-\dot{\imath\!\imath}{\omega}kT}=\mathbf{P}(t)e^{-\dot{\imath\!\imath}{\omega}kt}.
$$
Thus the columns of $\left(\mathbf{P}(t)e^{-\dot{\imath\!\imath}{\omega}kt}\right)$ again span an invariant subspace of $\mathscr{L}$ and a phase shift in $\mathbf{P}(t)$ by $e^{-\dot{\imath\!\imath}{\omega}kt}$ produces a purely imaginary shift in  the spectrum of $\mathbf{R}$ by $\dot{\imath\!\imath}{\omega}k$.

The relationship that the Floquet transformation bears to an underlying eigenvalue/invariant subspace problem is at the heart of the extension to a large scale setting.  Through an orthogonal change of basis, $\mathbf{R}$ may be taken to be upper triangular without loss of generality; that is, \eqref{FloqBVP} still holds with $\mathbf{P}(0)=\mathbf{P}(T)$  (but now $\mathbf{P}(0)\neq\mathbf{I}$ typically).  When $\mathbf{R}$ is upper triangular, the leading columns of $\mathbf{P}(t)$  span a family of nested (time-dependent) $T$-periodic subspaces that are invariant for $\mathscr{L}$.   
If $r$ denotes the desired reduction order, we make an additional assumption that $\mathbf{R}$ can be block-diagonalized into the direct sum of an $r\times r$ block and a complementary $(n-r)\times (n-r)$ block, or equivalently that the $r$ leading columns of $\mathbf{P}(t)$ span an invariant subspace for $\mathscr{L}$ while the trailing $n-r$ columns of $\mathbf{P}(t)$ span another invariant subspace for $\mathscr{L}$ having only a trivial intersection with the first.  
As a matter of practice we will only be interested in obtaining the $r$ leading columns of $\mathbf{P}(t)$.   The main algorithmic development that is pursued here seeks low ($r$) dimensional $T$-periodic invariant subspaces that are selectively targeted to produce effective modeling subspaces:
\begin{equation} \label{FloqArn}
\begin{aligned}
 \dot{\mathbf{P}}_r(t) &= \mathbf{A}(t)\mathbf{P}_r(t) -\mathbf{P}_r(t)\mathbf{R}_{r}\\ &\implies \ 
\mathscr{L}[\mathbf{P}_r](t)=-\dot{\mathbf{P}}_r(t)+ \mathbf{A}(t)\mathbf{P}_r(t)=\mathbf{P}_r(t)\mathbf{R}_{r}
\end{aligned}
\end{equation}
where $\mathbf{R}_{r}\in\mathbb{C}^{r\times r}$ is upper triangular.  The $r$ columns of the matrix-valued function, $\mathbf{P}_r(t)\in\mathbb{C}^{n\times r}$, are linearly independent $T$-periodic vector functions that span an $r$-dimensional invariant subspace of $\mathscr{L}$.  $\mathbf{P}_r(0)$ will not generally contain \emph{any} columns of the identity. We may assume the columns of $\mathbf{P}_r(0)$ to be orthonormal, but later it will be convenient to allow the columns of $\mathbf{P}_r(t)$ to contain select \emph{eigenvectors} of  $\mathscr{L}$.
We will call \eqref{FloqArn} a \emph{truncated Floquet decomposition} and view it as a step toward a \emph{partial Floquet transformation}. 

\subsection{Port-isolated {\small LTP} Systems and Partial Floquet Transforms}
\label{se:ltpsys}

Define $\mathscr{L}^{\star}=\frac{d}{dt}+\mathbf{A}^*(t)$, taking 
$\mathsf{Dom}(\mathscr{L}^{\star})\supset\mathsf{Dom}(\mathscr{L})$ dense in  $\mathscr{H}$, and observe that for   $\Bv(t)\in\mathsf{Dom}(\mathscr{L})\subset\mathscr{H}$ and $\Bw(t)\in\mathsf{Dom}(\mathscr{L}^{\star})\subset\mathscr{H}$, 
\[
\begin{aligned}
\langle\Bw,\mathscr{L}\Bv\rangle 
 = & \int_0^T \Bw(t)^*(-\dot{\Bv}(t)+\mathbf{A}(t)\Bv(t))\,dt \\
   = -(\Bw&(T)^*\Bv(T)-\Bw(0)^*\Bv(0))
   +\int_0^T \dot{\Bw}(t)^*\Bv(t)+\Bw(t)^*\mathbf{A}(t)\Bv(t)\,dt =\langle \mathscr{L}^{\star}\Bw, \Bv\rangle.
   \end{aligned}
\]
The boundary terms arising in the middle integration-by-parts step are zero due to periodicity. 
So, $\mathscr{L}^{\star}=\frac{d}{dt}+\mathbf{A}(t)^*$ is a true Hilbert space adjoint of $\mathscr{L}=-\frac{d}{dt}+\mathbf{A}(t)$. We use ``$*$" to denote ``conjugate-transpose" for complex vector/matrix-valued quantities and ``$\star$" to denote the $\mathscr{H}$-operator adjoint.  The similarity in notation is not accidental since for example, the mapping $\mathbf{v}\mapsto \mathbf{P}(t)\mathbf{v}$ in $\mathscr{H}$ (pointwise matrix-vector multiplication) has as its adjoint map in $\mathscr{H}$, $\mathbf{w}\mapsto \mathbf{P}(t)^*\mathbf{w}$ (pointwise matrix-vector multiplication by the conjugate transpose).

Complementary to \eqref{FloqBVP}, we define $\mathbf{Q}(t)=\mathbf{P}(t)^{-*}$ (i.e., the point-wise conjugate transpose of the matrix inverse, $\mathbf{P}(t)^{-1}$), and then, noticing that 
\begin{align*}
  \dot{\mathbf{Q}}(t)=&-\mathbf{P}(t)^{-*}\ \dot{\mathbf{P}}(t)^*\ \mathbf{P}(t)^{-*} \\
=& -\mathbf{P}(t)^{-*}(\mathbf{P}(t)^*\mathbf{A}(t)^* -\mathbf{R}^*\mathbf{P}(t)^*)\mathbf{P}(t)^{-*}\\
=& -\mathbf{A}(t)^*\mathbf{P}(t)^{-*} +\mathbf{P}(t)^{-*}\mathbf{R}^{*}
=-\mathbf{A}(t)^*\mathbf{Q}(t) +\mathbf{Q}(t)\mathbf{R}^*,
\end{align*}
we may determine that $\mathbf{Q}(t)$ satisfies:
\begin{equation} \label{FloqBVPdual}
\mathscr{L}^{\star}\mathbf{Q}(t)=\dot{\mathbf{Q}}(t) +\mathbf{A}(t)^*\mathbf{Q}(t) =\mathbf{Q}(t)\mathbf{R}^{*} 
\end{equation}
Thus, \emph{left invariant subspaces} of $\mathscr{L}$ can be  interpreted as (right) invariant subspaces of $\mathscr{L}^{\star}$. 
In particular, a left invariant subspace associated with $\mathbf{R}_{r}$ (as introduced in \eqref{FloqArn})
will be complementary to the (right) invariant subspace spanned by the columns of $\mathbf{P}_r(t)$ and can be characterized as a solution to 
\begin{equation} \label{redFloqBVPdual}
\begin{aligned}
    -\dot{\mathbf{Q}}_r(t)^{*} &= \mathbf{Q}_r(t)^{*}\mathbf{A}(t) -\mathbf{R}_{r}\mathbf{Q}_r(t)^{*}\\
    \implies &\ 
\mathscr{L}^{\star}\mathbf{Q}_r(t)=\dot{\mathbf{Q}}_r(t) +\mathbf{A}(t)^*\mathbf{Q}_r(t) =\mathbf{Q}_r(t)\mathbf{R}_{r}^{*} 
\end{aligned}
\end{equation}
However, if $\mathbf{P}_r(t)$ is interpretted as the leading $r$ columns of $\mathbf{P}$ satisfying \eqref{FloqBVP} with $\mathbf{R}$ block diagonal, then $\mathbf{Q}_r(t)$ in \eqref{redFloqBVPdual} will be similarly determined as the leading $r$ columns of $\mathbf{Q}$ satisfying \eqref{FloqBVPdual} (with same $\mathbf{R}$).  
Naturally we seek effective left and right modeling spaces spanned by the columns of $\BQ_r$ and $\BP_r$, \emph{without} assuming access to the full Floquet transform, i.e., without full knowledge of $\BP(t)$ or $\BQ(t)$. The following theorem characterizes these subspaces directly as invariant subspaces of $\Lop$. 
\begin{theorem} \label{thm:invSubspBases}
Let $\Lop$ be defined as above. If two matrix valued functions $\BP_r(t), \BQ_r(t)\in \bbC^{n\times r}$ satisfy $\mathsf{Ran}(\BP_r(t))\subset\mathsf{Dom}(\mathscr{L})$,
$\mathsf{Ran}(\mathbf{Q}_r(t))\subset\mathsf{Dom}(\mathscr{L})$, 
and
\begin{equation}
\label{leftrightInvariantSub}
\left.\begin{array}{c}
\mathscr{L}[\mathbf{P}_r](t)=\mathbf{P}_r(t)\mathbf{R}_{r}   \\[2mm]
\mathscr{L}^{\star}[\mathbf{Q}_r](t)=\mathbf{Q}_r(t)\mathbf{R}_{r}^{*}
\end{array}\ \right\}
\  \mbox{ for a given constant }  \mathbf{R}_r\in\mathbb{C}^{r\times r}  
\end{equation}
then $\mathbf{Q}_r(t)^{*}\mathbf{P}_r(t)=\mathbf{Q}_r(0)^{*}\mathbf{P}_r(0)$ for all $t\in[0,T]$. In particular,  $\mathbf{Q}(0)^{\star}\mathbf{P}(0)$ is nonsingular if and only if $\mathbf{Q}_r(t)^{*}\mathbf{P}_r(t)$ is nonsingular for each $t\in[0,T]$. 
\end{theorem}
\begin{proof}
  Let $\mathbf{v}(t), \mathbf{w}(t)$ be a pair of absolutely continuous periodic vector-valued functions on $[0,T]$ with values in $\mathbb{C}^r$ having square integrable components.  Observe
$$
\begin{array}{c}
\mathscr{L}[\mathbf{P}_r\mathbf{w}](t)=\mathscr{L}[\mathbf{P}_r](t)\cdot \mathbf{w}(t) - \mathbf{P}_r(t)\dot{\mathbf{w}}(t) \\[3mm]
\mathscr{L}^{\star}[\mathbf{Q}_r\mathbf{v}](t)=\mathscr{L}^{\star}[\mathbf{Q}_r](t)\cdot \mathbf{v}(t) + \mathbf{Q}_r(t)\dot{\mathbf{v}}(t)
\end{array}
$$
Then,
\begin{align*}
    \big\langle \mathbf{v},&\,(\mathbf{R}_{r}[\mathbf{Q}_r^{*}\mathbf{P}_r])\, \mathbf{w}\big\rangle  = 
    \big\langle (\mathbf{Q}_r\mathbf{R}_{r}^{*})\ \mathbf{v},\, \mathbf{P}_r\,\mathbf{w}\big\rangle \\[2mm]
    =&\big\langle  \mathscr{L}^{\star}[\mathbf{Q}_r\mathbf{v}]- \mathbf{Q}_r\dot{\mathbf{v}} ,\mathbf{P}_r\,\mathbf{w}\big\rangle \\[2mm]
   = & \big\langle  \mathbf{Q}_r\mathbf{v},\mathscr{L}[\mathbf{P}_r\,\mathbf{w}]\big\rangle - \big\langle \mathbf{Q}_r\dot{\mathbf{v}} ,\mathbf{P}_r\,\mathbf{w}\big\rangle \\[2mm]
   & = \big\langle  \mathbf{Q}_r\mathbf{v},\mathscr{L}[\mathbf{P}_r] \cdot \mathbf{w} - \mathbf{P}_r\dot{\mathbf{w}}\big\rangle - \big\langle \dot{\mathbf{v}} ,[\mathbf{Q}_r^{*}\mathbf{P}_r]\,\mathbf{w}\big\rangle \\[2mm] 
     &\quad = \big\langle  \mathbf{Q}_r\mathbf{v},\mathbf{P}_r\mathbf{R}_{r} \mathbf{w}\big\rangle -  \big\langle  \mathbf{Q}_r\mathbf{v},\mathbf{P}_r\dot{\mathbf{w}}\big\rangle + 
     \big\langle \mathbf{v} ,\mbox{\small $\frac{d}{dt}$}\!\left[\mathbf{Q}_r^{\star}\mathbf{P}_r\,\mathbf{w}\right]\big\rangle \\[2mm]
   &\qquad = \big\langle  \mathbf{v},[\mathbf{Q}_r^{*}\mathbf{P}_r]\mathbf{R}_{r} \mathbf{w}\big\rangle   + \big\langle \mathbf{v} ,\mbox{\small $\frac{d}{dt}$}\![\mathbf{Q}_r^{\star}\mathbf{P}_r]\,\mathbf{w}\big\rangle 
\end{align*}
Since $\mathbf{v}(t), \mathbf{w}(t)$ are arbitrarily chosen, we must have 
\begin{equation*}
0=\frac{d}{dt}\![\mathbf{Q}_r^{*}\mathbf{P}_r]-\mathbf{R}_{r}\,[\mathbf{Q}_r^{\star}\mathbf{P}_r](t)+[\mathbf{Q}_r^{\star}\mathbf{P}_r](t)\,\mathbf{R}_{r}
\quad \implies \quad 
0=\frac{d}{dt}\left(e^{-t\mathbf{R}_{r}}[\mathbf{Q}_r^{\star}\mathbf{P}_r]e^{t\mathbf{R}_{r}}\right)
\end{equation*}
for all $t\in[0,T]$.  This can be integrated and rearranged to find $$
\mathbf{Q}_r(t)^{*}\mathbf{P}_r(t)\,e^{t\,\mathbf{R}_r}=e^{t\,\mathbf{R}_r}\,\mathbf{Q}_r(0)^{*}\mathbf{P}_r(0), \mbox{ for all } t\in[0,T].
$$    
Now since the elements of $\mathsf{Ran}(\mathbf{P}_r(t))$ and $\mathsf{Ran}(\mathbf{Q}_r(t))$ are $T$-periodic, we have in particular that $\mathbf{Q}_r(T)^{*}\mathbf{P}_r(T)=\mathbf{Q}_r(0)^{*}\mathbf{P}_r(0)$, and consequently $e^{T\mathbf{R}_{r}}$ commutes with $\mathbf{Q}_r(0)^{*}\mathbf{P}_r(0)$.  Since $e^{t\mathbf{R}_{r}}=\left(e^{T\mathbf{R}_{r}}\right)^{t/T}$, we also can assert that for all $t\in[0,T]$, $e^{t\mathbf{R}_{r}}$ commutes with $\mathbf{Q}_r(0)^{*}\mathbf{P}_r(0)$ as well, and so,  $\mathbf{Q}_r(t)^{*}\mathbf{P}_r(t)=\mathbf{Q}_r(0)^{*}\mathbf{P}_r(0)$.
\end{proof}

The condition \eqref{leftrightInvariantSub} is satisfied when the columns of $\mathbf{P}_r$ and $\mathbf{Q}_r$ span complementary right and left invariant subspaces, respectively, of dimension $r$ for $\mathscr{L}$. The eigenvalues of $\BR_r$ are a subset of the spectrum of $\mathscr{L}$, with algebraic and geometric multiplicities relative to $\BR_r$ bounding those relative to $\mathscr{L}$ from below. Theorem \ref{thm:invSubspBases} suggests that one could capture  benefits of a Floquet transform, at least in part, by identifying left/right invariant subspaces of $\mathscr{L}$ that could then serve as effective modeling subspaces producing port-isolated time-periodic reduced order models. This is developed further in the following theorem. We denote in the usual way $L_2(\mathbb{C}^n)=\left\{  \mathbf{z}(t)\in\mathbb{C}^n\left|\, \int_0^{\infty}\mathbf{z}(t)^*\mathbf{z}(t)\,dt<\infty\right.\right\}$ and recall that elements of the periodic function space $\mathscr{H}$ are extended from $[0,T]$ to $[0,\infty)$ via periodicity with no change in notation.  
\begin{theorem} \label{petrGalerkInvSub}
Suppose \eqref{eq:FOM} is asymptotically stable, $1\leq r\leq n$, and matrix-valued functions $\BP_r(t), \BQ_r(t)\in \bbC^{n\times r}$ are given that satisfy \eqref{leftrightInvariantSub} with $\mathsf{Ran}(\BP_r(t))\cup\mathsf{Ran}(\mathbf{Q}_r(t))\subset\mathsf{Dom}(\mathscr{L})$, and $\mathbf{M}_r=\mathbf{Q}_r(0)^{*}\mathbf{P}_r(0)$ nonsingular. Define left/right modeling subspaces as:
\begin{equation}
    \begin{aligned}
        \mathscr{P}_r=&\left\{\mathbf{P}_r(t)\mathbf{v}_r(t)\,\left|\,\begin{matrix}
        \mathbf{v}_r(t)\mbox{ absolutely continuous}\\
            \mbox{with } \mathbf{v}_r ,\dot{\mathbf{v}}_r\in L_2(\mathbb{C}^r)
        \end{matrix} \right.\right\}\quad\mbox{and}\\
        \mathscr{Q}_r&=\left\{\mathbf{Q}_r(t)\mathbf{w}_r(t)\,\left|\, \mathbf{w}_r \in L_2(\mathbb{C}^r).
\right.\right\}
    \end{aligned}
\end{equation}
Assuming that $u\in L_2(\mathbb{C})$,  consider a reduced-order model for \eqref{eq:FOM} determined by the Petrov-Galerkin condition: 
\begin{equation} \label{petGalCond}
\fbox{$\begin{aligned}
    \mbox{Find }\mathbf{p}(t)\in \mathscr{P}_r \mbox{ such that }
    -\dot{\mathbf{p}}+\mathbf{A}(t)\mathbf{p}(t)+\mathbf{b}(t)&u(t)\ \perp \  \mathscr{Q}_r\mbox{ in }L_2(\mathbb{C}^n),\\
    \mbox{The associated output is }y_r(t)=\mathbf{c}(t)^*\mathbf{p}(t)&
\end{aligned}$}
\end{equation}
Then the reduced model specified by \eqref{petGalCond} is equivalent to a port-isolated \textsc{ltp} system given by: 
\begin{equation} \label{ProjLTPModel}
 \fbox{$\begin{aligned}
    \dot{\mathbf{z}}_r(t) &= \mathbf{R}_r \mathbf{z}_r(t) + \left(\mathbf{M}_r^{-1}\mathbf{Q}_r^{*}(t)\mathbf{b}(t)\right) {u}(t)\\
		{y}_r(t) &= \left(\mathbf{c}(t)^* \mathbf{P}_r(t)\right) \mathbf{z}_r(t).
    \end{aligned}$}
\end{equation}
\end{theorem}
\begin{proof} Since $u\in L_2(\mathbb{C})$, 
 we have $-\dot{\mathbf{p}}(t)+\mathbf{A}(t)\mathbf{p}(t)+\mathbf{b}(t)u(t)\in L_2(\mathbb{C}^n)$ for any $\mathbf{p}(t)\in \mathscr{P}_r$.  A solution to \eqref{petGalCond} will then exist if and only if $\mathbf{p}(t)=\mathbf{P}_r(t)\mathbf{z}_r(t)$ satisfies the orthogonality condition in \eqref{petGalCond} for some $\mathbf{z}_r\in L_2(\mathbb{C}^r)$, i.e., if and only if for all $\mathbf{w}_r \in L_2(\mathbb{C}^r)$ we have
    \begin{align*}
0=&\int_0^{\infty}\left(\mathbf{Q}_r(t)\mathbf{w}_r(t)\right)^*\left(-(\mathbf{P}_r(t) \mathbf{z}_r(t))^{\bullet}+\mathbf{A}(t)\mathbf{P}_r(t) \mathbf{z}_r(t)+\mathbf{b}(t)u(t)\right)\,dt \\
&=\int_0^{\infty}\left(\mathbf{Q}_r(t)\mathbf{w}_r(t)\right)^*\left(\mathscr{L}[\mathbf{P}_r](t)\,\mathbf{z}_r(t)-\mathbf{P}_r(t) \dot{\mathbf{z}}_r(t)+\mathbf{b}(t)u(t)\right)\,dt \\
&=\int_0^{\infty}\mathbf{w}_r(t)^*\left(\mathbf{Q}_r(t)^*\mathbf{P}_r(t)\left(\mathbf{R}_r \mathbf{z}_r(t) -\dot{\mathbf{z}}_r(t)\right) +\mathbf{Q}_r(t)^*\mathbf{b}(t)u(t)\right)\,dt \\
&=\int_0^{\infty}\mathbf{w}_r(t)^*\mathbf{M}_r\left(\mathbf{R}_r \mathbf{z}_r(t) -\dot{\mathbf{z}}_r(t) +\mathbf{M}_r^{-1}\mathbf{Q}_r(t)^*\mathbf{b}(t)u(t)\right)\,dt. 
   \end{align*}
   This is equivalent to  \eqref{ProjLTPModel} being satisfied $t$-almost everywhere.
\end{proof}

The success of \eqref{ProjLTPModel} as a reduced-order \textsc{ltp} model surrogate for the original system \eqref{eq:FOM} will depend on the extent to which the system response is captured by the response of a subsystem with dynamics restricted to the invariant subspace of $\mathscr{L}$ distinguished by \eqref{leftrightInvariantSub}.  This leads us to consider strategies for identifying effective invariant subspaces of $\mathscr{L}$ motivated by analogous successful methods for reducing \textsc{lti} systems; we  pattern our approach after the \emph{Dominant Pole Algorithm} of Rommes \cite{rommesDPA}.

\section{Determining Effective Invariant Subspaces
}\label{se:ltp2lti}

Consider the Floquet-transformed input/output system in \eqref{CBTpostFloquet}. For convenience in what follows we will assume that $\mathbf{R}$ is diagonalizable and the diagonalizing similarity transformation has been applied and absorbed into $\mathbf{P}(t)$ with no change in notation; $\mathbf{R}=\boldsymbol{\Lambda}$ is diagonal.  We denote the Floquet transformed input/output ports as $\bhat(t):=\BQ(t)^{*}\Bb(t)$ and $\chat(t):=\BP(t)^*\Bc(t)$. If we assume they are band limited (i.e represented by a finite Fourier series), then their Fourier expansions are given by:   
 \begin{equation}
\label{eq:FourB}
\begin{aligned}
\bhat(t):=\sum_{k=-K}^K \Bhat_{:,k} e^{\ii k t} =\begin{bmatrix}
    \Bhat_{:,-K} & \dots & \Bhat_{:,K}
\end{bmatrix}
\begin{bmatrix}
    e^{-\ii K\omega t} \\
    \vdots \\
    e^{\ii K\omega t}
\end{bmatrix}=:
\Bhat\, \BPsi_{-K:K}(t),
\end{aligned}
\end{equation}
\begin{equation}
\label{eq:FourC}
\begin{aligned}
\chat(t):=\sum_{k=-K}^K \Chat_{:,k} e^{\ii k\omega t}=\begin{bmatrix}
    \Chat_{:,-K} & \dots & {\Chat}_{:,K}
\end{bmatrix}
\begin{bmatrix}
    e^{-\ii K\omega t} \\
    \vdots \\
    e^{\ii K\omega t}
\end{bmatrix}
=:
{\Chat}\,\BPsi_{-K:K}(t).
\end{aligned}
\end{equation}

We refer to $K$ as the \emph{Fourier depth}, denoting the highest frequency represented in these expansions. For the rest of the paper we will make use of the vector of phase basis functions, which we define here:
\begin{definition}[Phase Basis Functions]\label{def:psi}
    For a given system base frequency, $\omega$, we define phase functions  $\psi_\ell = e^{\ii \ell \omega t}$ for $\ell=-2K,\ldots,2K$. The \emph{phase basis vector}, $\BPsi$, is then defined as 
    \begin{equation}
        \BPsi(t) = \begin{bmatrix}
            \psi_{-2K}(t) & \hdots & \psi_{2K}(t)
        \end{bmatrix}^T = \begin{bmatrix}
            e^{-\ii 2K\omega t} & \hdots & e^{\ii 2K\omega t}
        \end{bmatrix}^T
    \end{equation}
    We set the default range of $\BPsi$ to $-2K$ to $2K$, as that covers most purposes in this paper. If a different range is used, we specify it explicitly in the subscript, as in \eqref{eq:FourB} and \eqref{eq:FourC}.    
\end{definition}

We may lift the system to an \textsc{lti-mimo} (Linear Time-Invariant, Multiple Input/ Multiple Output) representation by absorbing $\BPsi_{-K:K}$ into our input and output. That is, letting $\widehat{\Bu}(t):= \BPsi_{-K:K}(t)u(t)$ and $y(t) = \BPsi_{-K:K}(t)^*\,\widehat{\By}(t)$, we recover an imbedded \textsc{lti} representation, as shown in Figure \ref{fig:LTIext}: 

\begin{figure}[ht!]
  \centering
\includegraphics[width=1.05\linewidth]{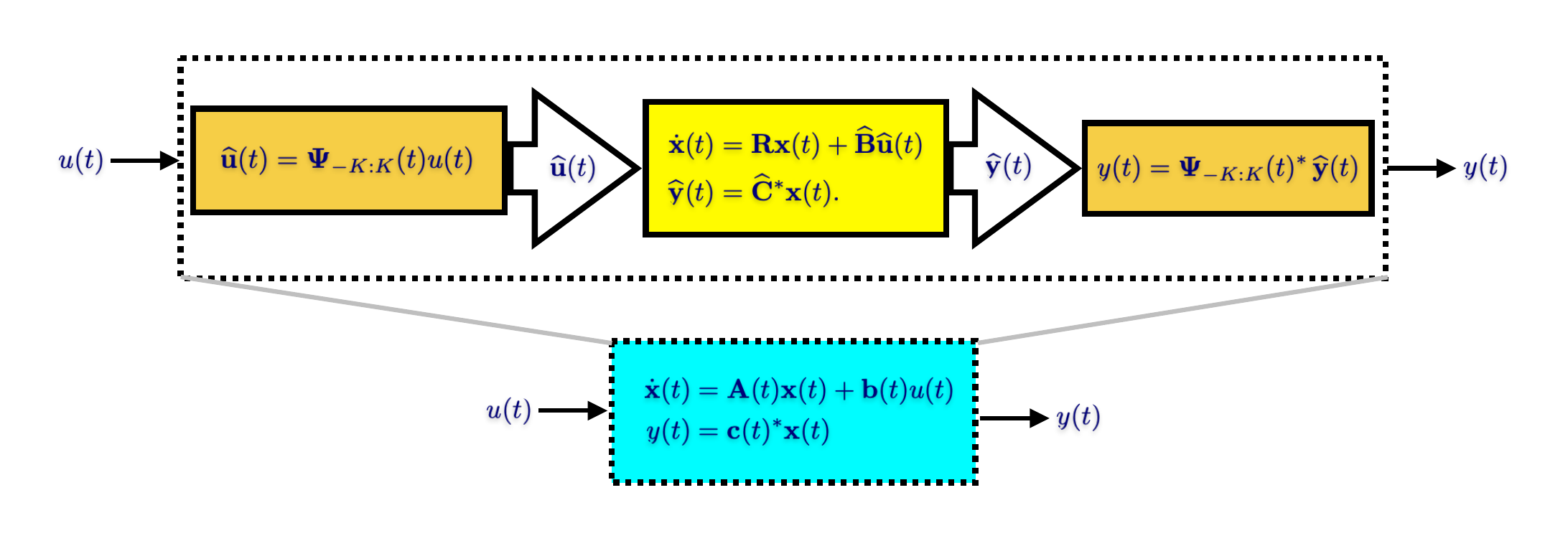}
\caption{Mapping an \textsc{LTP} system to an \textsc{LTI} extension}
\label{fig:LTIext}
    \end{figure}

We refer to the intermediate \textsc{lti} \textsc{mimo} system, $\left[ \begin{array}{c|c}
\BR &  \widehat{\BB} \\  \hline
\widehat{\BC}^* & 
\end{array} \right]$, an {\it{\textsc{lti} extension}}. The transfer function of the \textsc{lti} extension is $\BH_{ext}(s) = \widehat{\BC}^*\left(s\BI-\BR\right)^{-1}\widehat{\BB}$. For \textsc{lti} systems, the $H_\infty$ norm is the supremum of the transfer function's spectral norm along the imaginary axis: $\|\BH_{ext}\|_{H_\infty}:= \sup_{s\in \bbR}\|\BH_{ext}(\ii s)\|_{2}$. This norm is particularly useful because it provides a bound on the system's output energy: if $\widehat{\Bu}\in L_2$, then $\|\widehat{\By}\|_{L_2} \leq \|\BH_{ext}\|_{H_\infty}\|\widehat{\Bu}\|_{L_2}$. Hence, we can use the $H_\infty$ norm of the difference of transfer functions to evaluate the quality of a reduced-order model. In particular, if $\BH_{ext,r}(s)$ is the transfer function for a reduced-order model of the \textsc{lti} extension, then 
\begin{equation}
\label{eq:Herr}
    \|\widehat{\By}-\widehat{\By}_r\|_{L_2} \leq \|\BH_{ext}-\BH_{ext,r}\|_{H_\infty}\|\widehat{\Bu}\|_{L_2}.
\end{equation}
Recall that the actual output is recovered from the lifted signal via $y(t)-y_r(t) =  \BPsi_{-K:K}(t)^*(\widehat{\By}(t)-\widehat{\By}_r(t))$. Applying the triangle and Cauchy-Schwartz inequalities yields the bound
\begin{equation}
\label{eq:yerr}
    \|y-y_r\|_{L_2}^2 
    =
    \left\|\sum_{k=-K}^K e^{-\ii k\omega t}\left(\widehat{\By}_k-\widehat{\By}_{r,k}\right)\right\|_{L_2}^2 
    \leq
    \|\widehat{\By}-\widehat{\By}_r\|_{L_2}^2.
\end{equation}
Combining \eqref{eq:Herr} and \eqref{eq:yerr}, we conclude that model reduction techniques which produce accurate approximations to $\BH_{ext}(s)$ in $H_\infty$ will also produce accurate approximations to $y(t)$ in $L_2$. This relation serves as a theoretical basis for ranking different eigentriples of $\Lop$. To proceed, we express $\BH_{ext}(s)$ in pole-residue form and rank the terms based on how `dominant' they are.

\subsection{Dominant Pole Truncation}
\label{se:DPexpectations}
With diagonal $\BR$, say $\BR = \BLambda=\mathsf{diag}(\lambda_1,\lambda_2,\ldots,\lambda_n)$, we have
\begin{equation}
\label{eq:Hpoleres}
\BH_{ext}(s)=\widehat{\BC}^*(s\BI-\BLambda)^{-1}\widehat{\BB} =\sum_{j=1}^n \frac{\Theta_j}{s-\lambda_j}, \quad \text{where } \Theta_j=(\widehat{\BC}^*\Be_j)(\Be_j^*\widehat{\BB}).
\end{equation}
Dominant pole truncation is a model reduction technique that takes the terms with the largest individual norms and truncates the rest. With the $H_{\infty}$ norm, terms are ordered by $\frac{\|\Theta_j\|_{2}}{|\re(\lambda_j)|}$. As a consequence of Parseval's relation and the norm of outer products, $\|\Theta_j\|_2=\|\Bp_j(t)^*\Bc(t)\|_{L_2}\|\Bq_j(t)^*\Bb(t)\|_{L_2},$ where $\Bp_j(t)$ and $\Bq_j(t)$ are the $j^{th}$ columns of $\BP(t)$ and $\BQ(t)$, respectively. Therefore, we can establish a measure of importance -- or \emph{degree of dominance} -- for different eigentriples according to $\BH_{ext}$:  
\begin{equation}\label{eq:degdom}
\begin{aligned}
\BH_{ext}\text{ degree of dominance of }& \lambda_j = \frac{\|\Chat_{j,:}\|_{2}\|\Bhat_{j,:}\|_{2}}{|\re(\lambda_j)|}\\[2mm]
=&\frac{\|\Bp_j(t)^*\Bc(t)\|_{L_2}\|\Bq_j(t)^*\Bb(t)\|_{L_2}}{|\re(\lambda_j)|}.  
\end{aligned}
\end{equation}

Dominant pole truncation is a heuristic model and does not give any optimality guarantees. Its effectiveness depends on the location and residues of the poles. Taking $\BH_{ext,r}$ to be the transfer function induced by dominant pole truncation, 
\begin{equation}
    \|\BH_{ext}-\BH_{ext,r}\|_{H_\infty} \leq \sum_{j=r+1}^n\frac{\|\Theta_j\|_{2}}{|\re(\lambda_j)|}.
\end{equation}
This inequality suggests that the technique will be most effective when a cluster of poles are located close to the imaginary axis, while the remainder are significantly farther to the left in the complex plane.\\
Despite its heuristic nature, dominant pole truncation offers two practical advantages. First, it retains the poles of the original system, meaning that poles retain their physical interpretation and the reduced order model preserves stability.  Second, and more significantly in our setting, it aligns naturally with the structure of invariant subspaces. This makes it particularly suitable for \textsc{ltp} systems: performing a partial Floquet transform using the $r$ most dominant eigentriples yields the same reduced system as first computing the full Floquet transform and then applying dominant pole truncation (see Theorem~\ref{petrGalerkInvSub}). This efficient path to a reduced order model makes dominant pole truncation a fitting strategy in the context of time-periodic model reduction.

\subsection{An $\mathbf{H}_{ext}$ Proxy and the Harmonic Transfer Function}

As we discuss in \S\ref{se:DPA}, tools to identify dominant modes exist for \textsc{lti} systems, however these methods presume access to a transfer function. A challenge in this setting is that $\BH_{ext}$ can only be evaluated after performing a full Floquet transform. Consequently, it is necessary to replace $\BH_{ext}$ with a proxy function -- an accessible function with similar pole and residue characteristics. A natural starting point is to translate the parameters of an \textsc{lti} transfer function to their most intuitive \textsc{ltp} counterparts. For an \textsc{ltp-siso} system, \eqref{eq:FOM}, one might treat $\Cg_0(s) = \left\langle\Bc,(s-\Lop)^{-1}[\Bb]\right\rangle$ as a kind of transfer function. Unfortunately, this proxy has an incompatible pole-residue structure and may misclassify dominant poles as removable singularities. Nonetheless, further refinements led to a useful connection to the \emph{Harmonic Transfer Function} (\textsc{htf}) \cite{wereleyLinearTimePeriodic1991,wereleyAnalysisControlLinear}. Given an \emph{exponentially modulated periodic} input, 
\begin{equation}
    u(t) = e^{st}\sum_{m} u_m e^{\ii m\omega t},
\end{equation}
the steady state output of \eqref{eq:FOM} must be of the same form. The \textsc{htf}, $\CG(s)$, relates the input harmonics to the output harmonics. That is, 
\begin{equation}
\label{eq:yGu}
    y(t) = e^{st} \sum_{\ell}\left(\sum_{m}\CG_{\ell,m}(s)u_m \right)e^{\ii \ell \omega t}. 
\end{equation}
In what follows, we show that $\Cg_0$ equals the central component of the \textsc{htf} and similar expressions exist for all other components.  This provides (to the authors' knowledge) a novel representation of the \textsc{htf}. In any case, a proxy function, $\Cg$, is produced which represents the output harmonics when the input is a complex exponential signal, $u(t) = u_0e^{st}$. Beyond physical interpretability,  $\Cg$ also has a similar pole/residue structure as $\BH_{ext}$:


\begin{lemma}
    \label{lem:bcheck}
    Let $\bhat(t)$ be band limited with Fourier depth no greater than $K$. Then
        $$
        (s\bbI-\Lop)^{-1}[\Bb](t) = \BP(t)\sum_{k=-K}^K((s+\ii \omega k)\BI - \BR)^{-1}\Bhat_{:,k}e^{\ii k\omega  t}
        $$
    \end{lemma}
\begin{proof}
Recall that $\bhat(t) = \BP(t)^{-1}\Bb(t)$, $\chat(t) = \BP(t)^*\Bc(t)$, and $\Bhat$, $\Chat$ are their Fourier coefficients (\S \ref{se:ltp2lti}). Hence, $\Bb(t) = \BP(t)\bhat(t)$, implying
\begin{equation}
\Bb(t)=\BP(t)\widehat{\Bb}(t)\\
=(s\bbI-\Lop)\BP(t)\widecheck{\Bb}(t)\\
=\BP(t)\left[(s\BI-\BR)\widecheck{\Bb}(t)+\dot{\widecheck{\Bb}}(t)\right].
\end{equation}
By solving 
$\widehat{\Bb}(t) = (s\BI-\BR)\widecheck{\Bb}(t)+\dot{\widecheck{\Bb}}(t)$
for $\widecheck{\Bb}(t)\in\mathsf{Dom}(\mathscr{L})$ and taking a Fourier expansion we get the desired expression. 
\end{proof}

\begin{remark}
\label{rm:depth}
By repeating these arguments we can find a similar expression for $(s\bbI-\Lop)^{-*}[\Bc](t)$. Note then that we can find the Fourier depth, $K$, by counting the number of Fourier coefficients in $(s\bbI-\Lop)^{-1}[\Bb](t)$ and $(s\bbI-\Lop)^{-*}[\Bc](t)$.
\end{remark}

One may now see that the poles of $\Cg_0(s)$ -- and by extension, the eigenvalues of $\Lop$ -- are the eigenvalues of $\BR$ shifted by integer multiples of $\ii\omega$. Indeed, recall from \S\ref{sec2_problemSetting} that if $\Bp(t)$ and $\Bq(t)$ are right and left eigenfunctions associated with $\lambda$, then $\Bp(t)e^{-\ii k\omega t}$ and $\Bq(t)e^{-\ii k\omega t}$ are right and left eigenfunctions, respectively, associated with the shifted eigenvalue, $\lambda+\ii k$. The cancellation of the phase shifts implies that the specific branch of an eigentriple does not matter for the purposes of the Petrov-Galerkin projection, however Theorem \ref{thm:g0_wer} reveals that the presence of these shifts makes it difficult to perfectly replicate the degree of dominance asserted by $\BH_{ext}$. 

\begin{theorem} Let $\bhat(t)$ and $\chat(t)$ be band limited with Fourier depth no greater than $K$. Then
\label{thm:g0_wer}
    $$
    \Cg_{\ell}(s) := \left\langle\Bc \psi_\ell,(s\bbI-\Lop)^{-1}[\Bb]\right\rangle = \sum_{k=-K}^K \widehat{\BC}_{:,k-\ell}{}^* ((s+\ii k\omega)\BI-\BR)^{-1}\widehat{\BB}_{:,k}.
    $$
\end{theorem}
\begin{proof}
    Recall that $\psi_\ell(t) = e^{\ii \ell \omega t}$ and the inner product is given by
    \begin{equation}
    \label{eq:ip_def}
    \left\langle\Bc \psi_\ell,(s-\Lop)^{-1}[\Bb]\right\rangle = \frac1T \int_0^T e^{-\ii \ell \omega t}\Bc(t)^*\,(s-\Lop)^{-1}[\Bb](t)dt.
    \end{equation}
    The theorem follows from these steps: (1) Apply Lemma \ref{lem:bcheck} to $\langle\Bc e^{\ii \ell \omega t},(s-\Lop)^{-1}[\Bb]\rangle$.  (2) Write $\Bc$ as $\Bc(t)=\BQ(t)\chat(t)$ and use the fact that $\BQ(t)^* = \BP(t)^{-1}$ to remove $\BP$ and $\BQ$ from the inner product. (3) Take the Fourier expansion of $\chat$ and use the orthogonality of the phase basis functions to evaluate the inner product.  
\end{proof}

The entries of $\CG$ are described by 
\begin{equation}
\label{eq:werHTF_lm}
    \CG_{\ell,m}(s) = \sum_{k=-K}^K \widehat{\BC}_{:,k-\ell}{}^* ((s+\ii k\omega)\BI-\BR)^{-1}\widehat{\BB}_{:,k-m}.
\end{equation}
Theorem \ref{thm:g0_wer} shows that $\Cg_\ell$ describes the $\ell^{th}$ component of central column of the \textsc{htf} ($\CG_{\ell,0}(s) = \Cg_\ell(s)$). In particular, $\Cg_0(s)$ expresses $\CG_{0,0}(s)$. For diagonal $\BR = \BLambda$, the pole residue form of $\Cg_0$ is: 
\begin{equation}
\label{eq:g0_pr}
\Cg_0(s) = \sum_{j=1}^n\sum_{k=-K}^K \frac{\overline{\Chat}_{j,k}\Bhat_{j,k}}{s-\lambda_j+\ii k}.
\end{equation}
Investigating this expression, one can see that if $\chat_j$ and $\bhat_j$ share no common harmonics, then the corresponding $\lambda_j$ and all of its shifts will be removable singularities of $\Cg_0$. This is true even for highly dominant poles: they will be invisible to $\Cg_0$ if the Floquet transformed ports have mismatched harmonics. \\
Note that what may be a removable singularity for $\Cg_0$ could be a dominant pole for a different $\Cg_{\ell}$. However, this observation does not fix the problem -- any scalar component will suffer from a potential mismatch of harmonics. To address this, we broaden our perspective: we stack all the components, $\Cg_\ell$, into a vector -- effectively composing the principal column of the \textsc{htf}, $\CG_{:,0}$. 
\begin{definition}[Principal Harmonics Vector (PHV)]
    Let $\bhat$ and $\chat$ be band limited with Fourier depth $K$, and take $\BPsi$ from Def \ref{def:psi}. We define the \emph{Principal Harmonics Vector} by
$$
\Cg(s):= \langle \Bc \BPsi^T,(s\bbI-\Lop)^{-1}[\Bb]\rangle = \frac{1}{T} \int_0^T \overline{\BPsi}(t)\, \Bc(t)^*(s\bbI-\Lop)^{-1}[\Bb](t)dt.
$$
\end{definition}
Observe that if the system has Fourier depth $K$, then $\CG_{\ell,0} = 0$ for $|\ell| > 2K$, and so this expression accounts for all non-zero components in the principal column of the \textsc{htf}, $\CG_{:,0}$. Examining \eqref{eq:yGu}, one sees that $\Cg$ captures the harmonic content of the output signal given a complex exponential, $u=e^{st}$, as input. Thus, $\Cg$ reflects a sensitivity to dominant modes whose energy is present anywhere in the frequency domain. \\

We now turn to the pole-residue form of $\Cg$ to better understand which terms are dominant. In \eqref{eq:g0_pr}, we expanded the pole-residue form of the central element of $\Cg$. The pole-residue for any component is simlar, $\Cg_\ell (s) = \sum_{j=1}^n\sum_{k=-K}^K \frac{\overline{\Chat}_{j,k-\ell}\Bhat_{j,k}}{s-\lambda_j+\ii k}$. Concentrating on a fixed value for $k$, the contributions from the $k^{th}$ and adjacent components of the PHV are
\begin{equation}
    \Cg_k = \sum_{j=1}^n \frac{\overline{\Chat}_{j,0}\Bhat_{j,k}}{s-\lambda_j+\ii k}, \quad \Cg_{k-1} = \sum_{j=1}^n \frac{\overline{\Chat}_{j,1}\Bhat_{j,k}}{s-\lambda_j+\ii k}, \quad  \Cg_{k+1} = \sum_{j=1}^n \frac{\overline{\Chat}_{j,-1}\Bhat_{j,k}}{s-\lambda_j+\ii k}.
\end{equation}
Extending this pattern, the $2K+1$ components of $\Cg$ centered around $k$ are given by
\begin{equation}
    \begin{bmatrix}
        \Cg_{k-K}\\
        \vdots \\
        \Cg_{k+K}
    \end{bmatrix}
    =
    \sum_{j=1}^n
    \frac{\Bhat_{j,k}}{s-\lambda_j+\ii k}\begin{bmatrix}
        \overline{\Chat}_{j,K}\\
        \vdots \\
        \overline{\Chat}_{j,-K}
    \end{bmatrix} = \sum_{j=1}^n \mathtt{flip}(\Chat_{j,:})^*\frac{\Bhat_{j,k}}{s-\lambda_j+\ii k},
\end{equation}
where \texttt{flip} reverses the order of the entries in the row vector $\Chat_{j,:}$. \\
This accounts for $2K+1$ of the $4K+1$ components in $\Cg(s)$. The remaining components are zero due to band limitation: $\Chat_{j,\widehat{k}} = 0$ for $|\widehat{k}|>K$. Importantly, neither the zero padding nor the \texttt{flip} operation alters the norm of each residue, and so dominance is ordered by:
\begin{equation}
\label{eq:gdom}
\Cg\text{ degree of dominance of } \lambda_j = \frac{\displaystyle \max_{k}\|\widehat\BB_{j,k}\,\overline{\widehat\BC}_{j,:}\|_2}{|\re{\lambda_j}|} =\frac{\displaystyle \|\overline{\widehat\BC}_{j,:}\|_2\|\widehat\BB_{j,:}\|_{\infty}}{|\re{\lambda_j}|}.  
\end{equation}

While \eqref{eq:gdom} does not exactly match \eqref{eq:degdom}, we care less about the actual value for the degrees of dominance and more about whether the two notions produce a similar ordering of poles. For this purpose we find \eqref{eq:gdom} satisfactory. \\

Beyond producing a more agreeable notion of dominance, the PHV also has connections to the $H_2$ norm for \textsc{ltp} systems.  Magruder et al. \cite{magruderLinearTimeperiodicDynamical2018} exemplified how the $H_2$ norm for \textsc{ltp} systems (defined in terms of the impulse response) can be expressed by treating $\Cg$ as an \textsc{lti} transfer function and taking its $H_2$ norm. This connection further supports the utility of the PHV.

\section{The Dominant Pole Algorithm}
\label{se:DPA}

We have conceptualized the `importance' of various eigenmodes -- \eqref{eq:degdom} and \eqref{eq:gdom} -- now we discuss our method of finding them. The Dominant Pole Algorithm (\textsc{dpa}) was originally developed to compute dominant eigenmodes of \textsc{lti} systems by applying Newton’s method to the reciprocal of the system's transfer function \cite{martinsComputingDominantPoles1996,rommesDPA}. In the following we describe the translation to the \textsc{ltp} setting.

\subsection{{\small LTP}-{\small DPA}}
The poles of $\Cg(s)$ are the $s\in\bbC$ where $\|\Cg(s)\|_{2}\ra \infty$. Equivalently, these are the zeros of the reciprocal function $\frac{1}{\|g(s)\|_2}$. Motivated by the classical Dominant Pole Algorithm (\textsc{dpa}), we adapt the method to the \textsc{ltp} setting by applying Newton’s method to the scalar objective function: $\frac{1}{\|g(s)\|_2}.$
\begin{equation}
\label{eq:gcol_newton}
    s_{k+1} 
    =
    s_k-\frac{\frac{1}{\|\Cg(s_k)\|_2}}{\frac{d}{ds}\frac{1}{\|\Cg(s_k)\|_2}} 
    = 
    s_k + \frac{\|\Cg(s_k)\|_2}{\frac{d}{ds}\|\Cg(s_k)\|_2}.
\end{equation}
Use 
\begin{equation}
    \frac{d}{ds}\|\Cg(s)\|_2 = \frac{\Cg(s)^*\frac{d}{ds}\Cg(s)}{\|\Cg(s)\|_2} \text{ and } \frac{d}{ds}\Cg(s) = -\langle \Bc\Bpsi^T,(s\bbI-\Lop)^{-2}\Bb\rangle
\end{equation}
to rewrite \eqref{eq:gcol_newton}.
\begin{equation}
    s_{k+1} = s_k + \frac{\|\Cg(s_k)\|_2}{\frac{\Cg(s)^*\frac{d}{ds}\Cg(s)}{\|\Cg(s)\|_2}} = s_k+\frac{\Cg(s_k)^*\Cg(s_k)}{\Cg(s_k)^*\Cg'(s_k)} = s_k - \frac{\Cg(s_k)^*\langle \Bc\BPsi^T,(s_k\bbI-\Lop)^{-1}\Bb\rangle}{\Cg(s_k)^*\langle \Bc\BPsi^T,(s_k\bbI-\Lop)^{-2}\Bb\rangle}
\end{equation}
Now substitute $\Balpha_k=\Cg(s_k)$, $\Bb(t) = (s_k\bbI-\Lop)\Bv_k(t)$ and $\Bc(t)\BPsi(t)^T\Balpha_k = (s_k\bbI-\Lop)^\star\Bw_k(t)$.
\begin{equation}
\begin{aligned}
    s_{k+1} &=  s_k - \frac{\langle (s_k-\Lop)^\star\Bw_k,\Bv_k\rangle}{\langle \Bw_k,\Bv_k\rangle} \\
    &= \frac{s_k\langle \Bw_k,\Bv_k\rangle - s_k\langle \Bw_k,\Bv_k\rangle+\langle \Bw_k,\Lop\Bv_k\rangle}{\langle \Bw_k,\Bv_k\rangle} = \frac{\langle \Bw_k,\Lop\Bv_k\rangle}{\langle \Bw_k,\Bv_k\rangle}.
\end{aligned}
\end{equation}
This leads to the \textsc{ltp} Dominant Pole Algorithm (\textsc{ltp}-\textsc{dpa}), Algorithm \ref{alg:ltpdpa}.

\begin{algorithm} [h!]
\caption{{\small LTP}-{\small DPA} }
\label{alg:ltpdpa}
\textbf{Input:} $\BA(t),\Bb(t),\Bc(t)$, initial guess -- $s_0$, tolerance -- $\varepsilon$.\\
Until convergence:
\begin{enumerate}
    \item Solve: \\
    $(s_k\bbI-\Lop) {\Bv_k(t)}=\Bb(t),$\\
    $\Balpha_k = \langle \Bc\BPsi^T,\Bv_k\rangle$, \\
    $(s_k\bbI-\Lop)^* \Bw_k(t)=\Bc(t)\BPsi(t)^T\Balpha_k$ 
    \item Update:\\
    $
    s_{k+1} = \frac{\langle\Bw_k,\Lop\Bv_k\rangle}{\langle\Bw_k,\Bv_k\rangle}
    $
     
    \item Check for convergence:\\
    If $\|\Lop\Bv_k-s_{k+1}\Bv_k\|_{L_2}<\varepsilon\|\Bv_k\|_{L_2}$, stop.
    \end{enumerate}
    \textbf{Return:} $\lambda = s_{k+1}$, $\Bp(t)=\Bv_k(t),$ $\Bq(t)=\frac{\Bw_k(t)}{\Bv_k(0)^*\Bw_k(0)}$
\end{algorithm}

\subsection{Improvements to {\small LTP}-{\small DPA} : Subspace Acceleration }
\label{se:SADPA}
In Algorithm~\ref{alg:ltpdpa}, $\Bv_k(t)$ and $\Bw_k(t)$ are used for a single iteration and then discarded. We can make convergence more precise by using historical information to build search spaces, $\BV(t)$ and $\BW(t)$. This idea, known as the Subspace Accelerated Dominant Pole Algorithm (\textsc{sadpa}), was introduced by Rommes in the \textsc{lti} setting \cite{rommesEfficientComputationTransfer2006}. Our translation to the \textsc{ltp} setting is mostly straightforward:
After finding $\Bv_k(t)$ and $\Bw_k(t)$, we use the inner product definition \eqref{eq:ip_def} to orthogonalize each against the columns of $\BV(t)$ and $\BW(t)$, respectively. Classical Gram-Schmidt is applied with reorthogonalization to ensure numerical stability. Once new vectors orthogonalized against previous ones and normalized, they  are appended as columns to $\BV(t)$ and $\BW(t)$. These subspaces are then used as modeling bases to build low rank \textsc{lti simo} approximations to $\Cg$. Specifically, at each iteration we construct 
\begin{equation}
\begin{aligned}
\widetilde{\Bh}(s) &= \Ctil{}^*(s\Etil - \Atil)^{-1}\btil, \\ \text{ where }
    \Atil = \langle \BW,\Lop\BV\rangle, \ &\Etil = \langle \BW,\BV\rangle, \ \btil = \langle \BW,\Bb\rangle, \ \Ctil = \langle \BV,\Bc\BPsi^T\rangle. 
\end{aligned}
\end{equation}
As the subspaces grow, $\widetilde{\Bh}$ becomes an increasingly accurate approximation of the PHV. We then rank the poles of $\widetilde{\Bh}$ according to the dominance criterion from \eqref{eq:gdom}, and use the most dominant one as the next shift, $s_k$. 

\subsection{Improvements to {\small LTP}-{\small DPA}: Deflation}
In order to search for multiple poles, we add a deflation procedure. Suppose we have converged to the eigentriple $(\lambda_1,\Bp_1(t),\Bq_1(t))$ of $\Lop$. Define $\Bb_{new}(t)=\Bb(t)-\Bp_1(t)\Bq_1(t)^*\Bb(t)$. Note that in the return statement of Algorithm \ref{alg:ltpdpa} we scale $\Bq_1(t)$ to ensure that $\Bq_1(0)^*\Bp_1(0) = 1$. By Theorem \ref{thm:invSubspBases} this ensures that $\Bq_1(t)^*\Bp_1(t) = 1$ for all $t$. Hence, 
\begin{equation}
\Bq_1(t)^*\Bb_{new}(t)=\Bq_1(t)^*\Bb(t)-\Bq_1(t)^*\Bp_1(t)\Bq_1(t)^*\Bb(t)=0,\end{equation}
and the adjusted residue will be zero. Since the left and right eigenspaces of the shifted pole, $\lambda_1+\ii\omega k$, are spanned by $e^{-\ii\omega k t}\Bq_1$ and $e^{-\ii\omega k t}\Bp_1$ (respectively), this deflation scheme zeroes out the residues of the entire family $\lambda_1+\ii\omega\bbZ$. Consequently, that $\lambda_1$ and all of its shifts will not be detected in future iterations.\\

Moreover, the deflation will not affect the residues of other poles. Let $\Bq_2(t)$ be a left eigenfunction corresponding to a different eigenvalue of $\BR$, $\lambda_2$, then 
\begin{equation}
\Bq_2(t)^*\Bb_{new}(t) = \Bq_2(t)^*\Bb(t)-\Bq_2(t)^*\Bp_1(t)\Bq_1(t)^*\Bb(t)=\Bq_2(t)^*\Bb(t),
\end{equation}
since left and right eigenfunctions are pointwise bi-orthogonal. A similar argument can be used to deflate $\Bc(t)$.
In the \textsc{lti} case, one can use the second most dominant pole of $\htil$ as the new shift. In the \textsc{ltp} setting the second most dominant pole of $\htil$ could very well be a shift of the recently found pole. Hence, after convergence to a pole, we recommend re-evaluating $\htil(s)$ with the deflated ports and using the dominant pole of the deflated system as the new shift.  

The modifications are summarized in Algorithm \ref{alg:ltpSADPA}.

\begin{algorithm} [h!]
\caption{{\small LTP}-{\small SADPA} (Subspace Accelerated Dominant Pole Algorithm)}
\label{alg:ltpSADPA}
\textbf{Input:} $\BA(t),\Bb(t),\Bc(t)$, initial guess(es) -- $s_0$, and $n_{want}$.

While $n_{found}<n_{want}$:
\begin{enumerate}
    \item Solve: \\
    $(s_k\bbI-\Lop) {\Bv_k(t)}=\Bb(t),$\\
    $\Balpha_k = \langle \Bc\BPsi^T,\Bv_k\rangle$, \\
    $(s_k\bbI-\Lop)^* \Bw_k(t)=\Bc(t)\BPsi(t)^T\Balpha_k$
    \item Build search spaces:\\
    $\BV(t) = \text{orth}([\BV(t)\ \Bv_k(t)]) $ and $\BW = \text{orth}([\BW(t)\ \Bw_k(t)])$
    \item Project $\Cg$:\\
    $\Atil = \langle \BW,\Lop\BV\rangle, \quad \Etil = \langle \BW,\BV\rangle, \quad \btil = \langle \BW,\Bb\rangle, \quad \widetilde{\BC} = \langle \BV,\Bc\BPsi^T\rangle$
    \item Search for poles:\\
    $\tilde\BX,\tilde\Lambda,\tilde\BY$ $\leftarrow$ find dominant poles of $\widetilde{\Bh}(s)=\widetilde{\BC}\,^*\left(s\Etil-\Atil\right)^{-1}\btil$
    \item Approximate eigentriple:\\
    $\tilde{\Bp}(t)=\BV(t)\tilde{\Bx}_1/\|\BV(t)\tilde{\Bx}_1\|$, $\tilde{\Bq}(t)=\BW(t)\tilde{\By}_1$
    \item Check for convergence.\\
    If $\|\Lop\tilde{\Bp}(t)-\tilde{\lambda}_1\tilde{\Bp}(t)\|<\varepsilon$:
    \begin{enumerate}
        
        \item Update: $\tilde{\Bq}(t)=\frac{\tilde{\Bq}(t)}{\tilde{\Bp}(0)^*\tilde{\Bq}(0)}$, $\Lambda=[\Lambda, \tilde{\lambda}_1]$, $\BP(t) = [\BP(t),\tilde{\Bp}(t)]$, $\BQ(t) = [\BQ(t), \tilde{\Bq}(t)]$, $n_{found}=n_{found}+1$
        \item Deflate: $\Bb(t) = \Bb(t) - \tilde{\Bp}(t)\tilde{\Bq}(t)^*\Bb(t)$, $\Bc(t) = \Bc(t) - \tilde{\Bq}(t)\tilde{\Bp}(t)^*\Bc(t)$
        \item If $n_{found}<n_{want}$: Rebuild $\widetilde{\Bh}(s)$ and find dominant pole (Re-run steps 3 and 4 here).  Set $s_{k+1}=\tilde{\lambda}_1$.
        \end{enumerate}
    Else: $s_{k+1}=\tilde\lambda_1$
    \end{enumerate}
    \textbf{Return:} $\BQ(t)$, $\BLambda$, $\BP(t)$ 
\end{algorithm}

\section{An Illustrative Example}
\label{se:numex}
To illustrate concepts discussed above, we introduce a simple \textsc{ltp} system. 
From \eqref{FloqBVP}, we have $\BA(t) = \dot{\mathbf{P}}(t)\mathbf{P}(t)^{-1}+\mathbf{P}(t)\mathbf{R}\mathbf{P}(t)^{-1}$. Through an intelligent choice of $\BP$, we can write down $\BP^{-1}$ explicitly:  
\begin{equation}
\centering
\label{eq:exP}
\mathbf{P}(t)=\begin{bmatrix}
    1 & \sin(t) & & \\
     & 1 & 0  & &\\
     & & \ddots & \ddots &
\end{bmatrix}\quad \Ra \quad 
\mathbf{P}(t)^{-1}=\begin{bmatrix}
    1 & -\sin(t) & & \\
     & 1 & 0  & &\\
     & & \ddots & \ddots &
\end{bmatrix}.
\end{equation}
This allows us to choose $\BR$ as we wish and experiment with the spectrum of $\Lop$. We will specify the values of $\BR$ as well as $\Bb(t)$ and $\Bc(t)$ keeping in mind the comments from \S\ref{se:DPexpectations}. We choose $\BR\in\bbR^{1000\times1000}$ with 
\begin{itemize}
    \item 10 eigenvalues logarithmically spaced between $-10^{-4}$ and $-1$ and
    \item 990 other eigenvalues logarithmically spaced between $-10^3$ and $-10^6$.
\end{itemize}
In \textsc{matlab} notation, $\mathtt{ R = -diag([logspace(-4,0,10)), logspace(3,6,990)])}$.\\

We set $\Bb$ and $\Bc$ to be vectors of ones. Note that after the Floquet transform, the ports will no longer be constant: 
\begin{equation}     
\widehat{\Bb}(t)=\BP(t)^{-1}\begin{bmatrix}
        1\\
        1\\
        \vdots
    \end{bmatrix}=\begin{bmatrix}
    1-\sin(t)\\
    1\\
    1-\sin(t)\\
    1\\
    \vdots
\end{bmatrix}
\text{ and }\widehat{\Bc}(t)=\BP(t)^*\begin{bmatrix}
        1\\
        1\\
        \vdots
    \end{bmatrix}=\begin{bmatrix}
    1\\
    1+\sin(t)\\
    1\\
    1+\sin(t)\\    
    \vdots
\end{bmatrix}.
\end{equation}
This choice of $\Bb(t)$ and $\Bc(t)$ guarantees that the residues for all poles have the same norm. Consequently, dominance (according to both $\BH_{ext}$ and $\Cg$) is governed solely by proximity to the imaginary axis.

 We use the \textsc{trigfun} package within the \textsc{Chebfun} toolbox to represent continuous time variables. \textsc{Chebfun} uses Fourier spectral collocation with variable mesh sizes to represent operators \cite{Driscoll2014}. The results shown here come from tinkering with \textsc{Chebfun} to fix the mesh size (32 collocation points) for faster computation. We compensate for various effects of discretization (such as spurious eigenvalues) but in the interest of brevity, we do not elaborate on these matters here. For a more detailed discussion of how trigonometric collocation affects the algorithms presented see \cite{benderReductionPeriodicSystems}. For a more general exposition of spectral collocation, see \cite{boydChebyshevFourierSpectral}.

Figure \ref{fig:dpa_eig} shows how \textsc{ltp}-\textsc{dpa} (Algorithm \ref{alg:ltpdpa}) performs on our example when given an initial guess, $s_0 = -0.1$, that is far from the most dominant pole, $\lambda_1 = -10^{-4}$.

\begin{figure} [h!]
    \centering
    \includegraphics[width=0.8\linewidth]{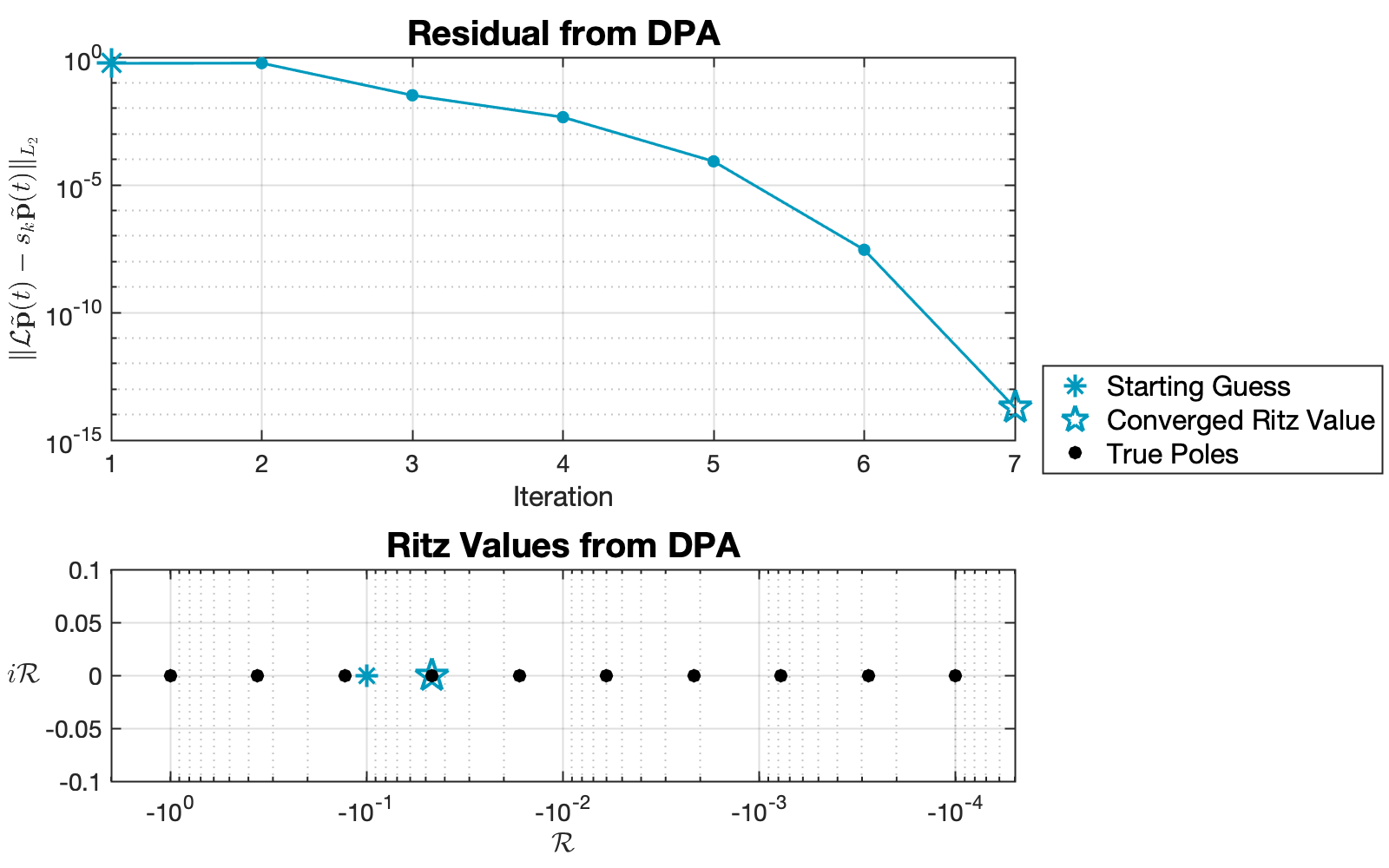}
    \caption{Convergence of Algorithm~\ref{alg:ltpdpa} using a poor initial guess of $s_0 = -0.1$. Tolerance set to $10^{-8}$.}
    \label{fig:dpa_eig}
\end{figure}
As explained earlier, dominance is determined solely by proximity to the imaginary axis. It is therefore encouraging that the iterates of \textsc{ltp}-\textsc{dpa} move to the right of the initial guess, however the algorithm finds a pole far from what what we desire. \\ 

Figure \ref{fig:sadpa_eig} demonstrates the benefits of adding deflation and subspace acceleration.

\begin{figure}[ht!]
    \centering
    \includegraphics[width=0.8\linewidth]{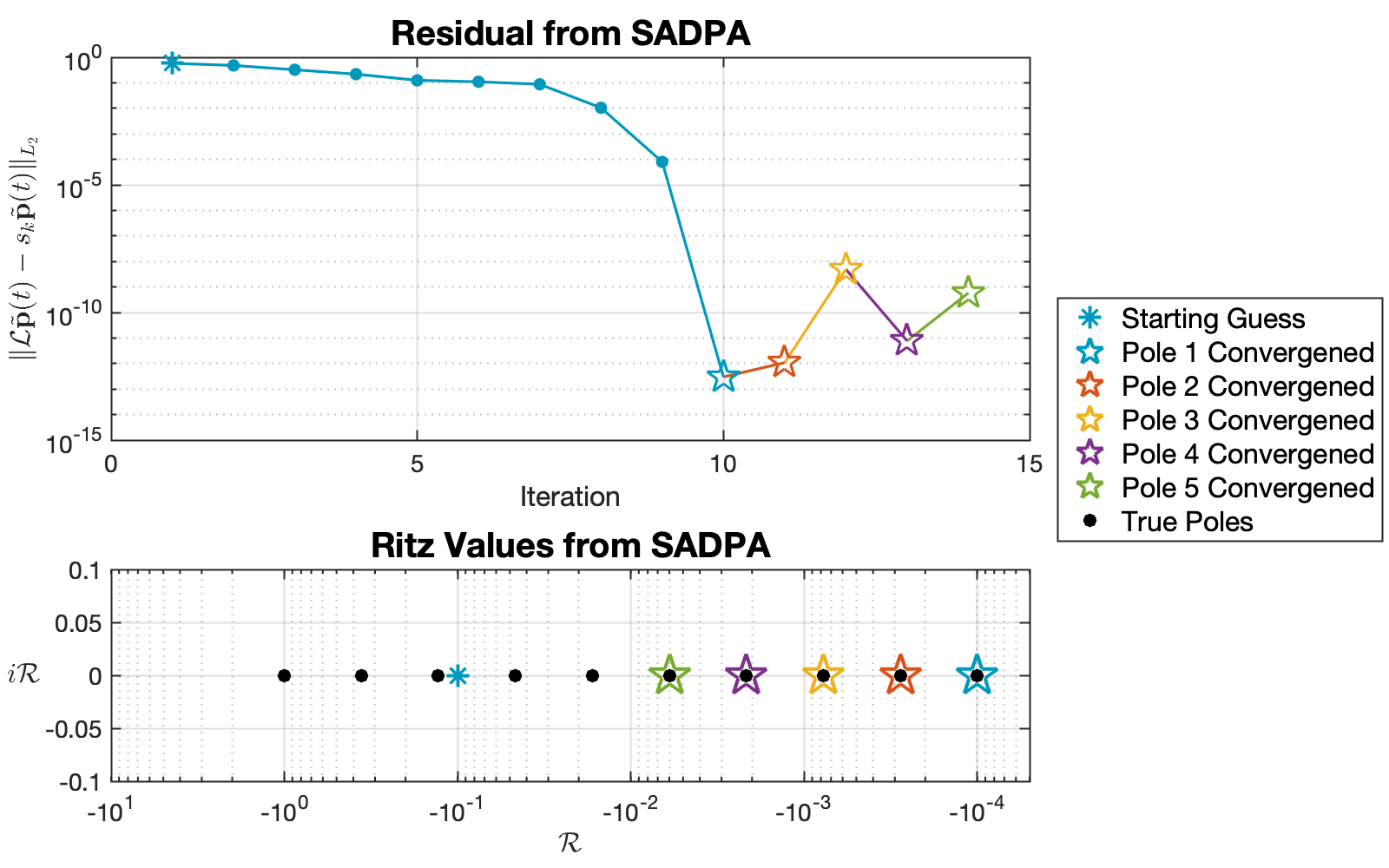}
    \caption{Convergence of Algorithm~\ref{alg:ltpSADPA} using a poor initial guess of $s_0 = -0.1$. Tolerance set to $10^{-8}$.}
    \label{fig:sadpa_eig}
\end{figure}

Unlike \textsc{ltp}-\textsc{dpa}, the algorithm \textsc{ltp}-\textsc{sadpa} is able to bypass less dominant poles and converge directly to the most dominant mode. After identifying the first dominant pole, subsequent poles are typically found in fewer iterations. In our test case using a bad initial guess, \textsc{sadpa} required only 14 total iterations to identify the five most dominant poles. This efficiency has significant payoff for large-scale systems that can be well-approximated by a small number of modes. The number of required backsolves to construct a reduced-order, port-isolated \textsc{ltp} model (as in \eqref{eq:ROM}) is dramatically lower than what is needed to first compute the full Floquet transform and then perform model reduction. \\

In \S\ref{se:DPexpectations}, we discussed conditions under which dominant pole truncation performs well and designed our example with these in mind. Figure~\ref{fig:simpexMR} confirms the effectiveness of the resulting reduced-order model.

\begin{figure}[ht!]
  \centering
  \begin{subfigure}[b]{0.45\textwidth}
    \includegraphics[width=1.05\linewidth]{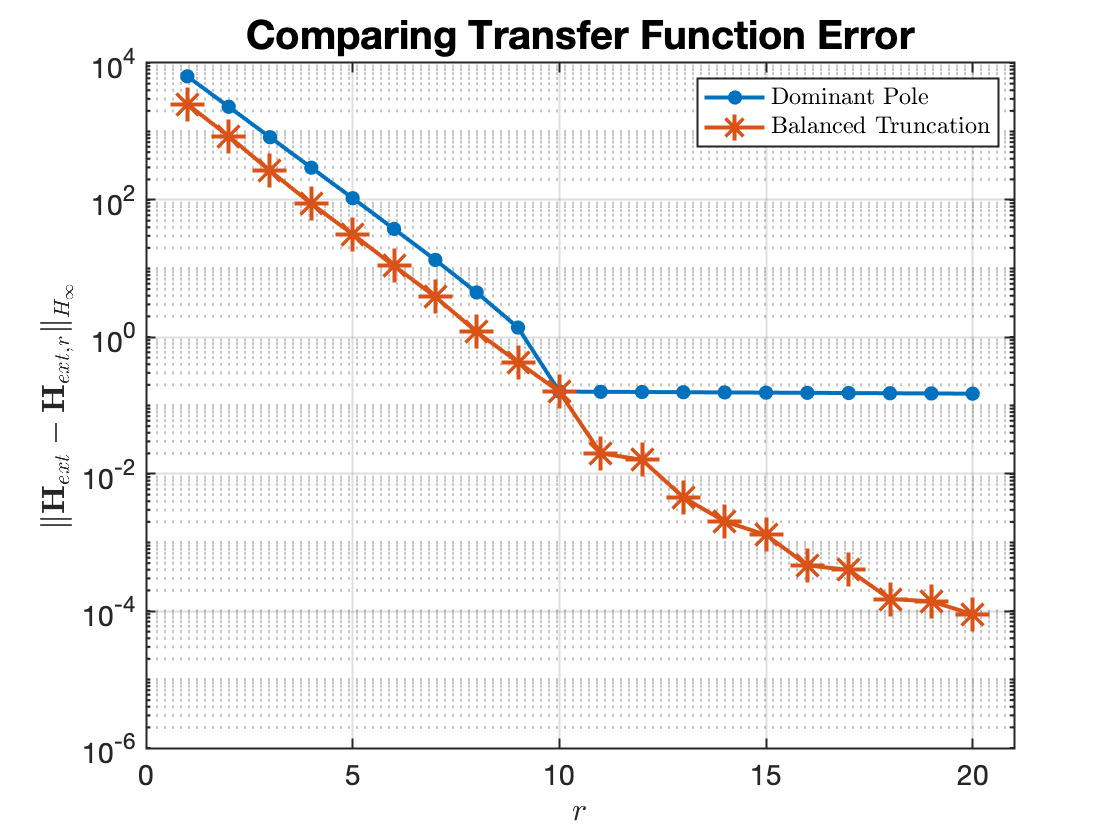}
    \caption{Dominant pole truncation performs similarly to Balanced Truncation -- a gold standard for \textsc{lti} Model Reduction -- for the first 10 orders.}
  \end{subfigure}
  \hfill
  \begin{subfigure}[b]{0.51\textwidth}
    \includegraphics[width=1.05\linewidth]{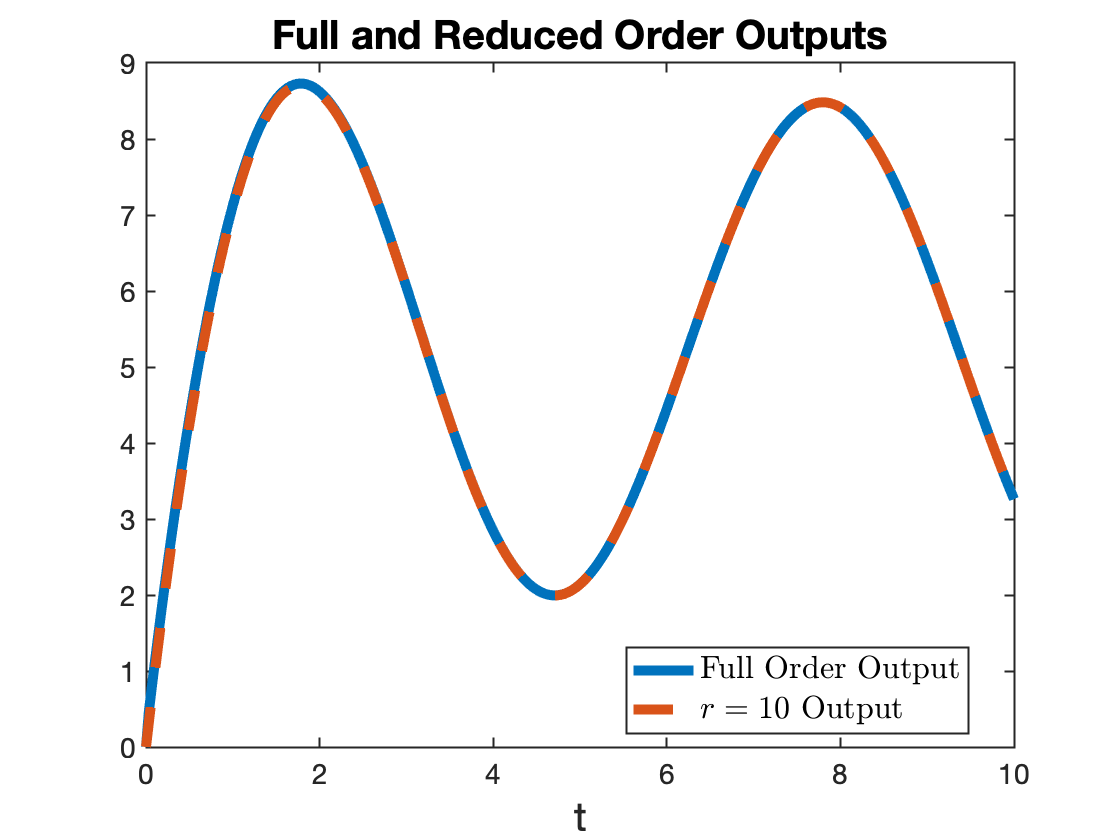}
    \caption{With $u(t) = e^{-t}$, output from the $r=10$ model nearly matches output from the original model with order $n=1000$. }
  \end{subfigure}
  \caption{Model Reduction Performance for Illustrative Example}
  \label{fig:simpexMR}
    \end{figure}

Although Balanced Truncation ultimately achieves lower error than dominant pole truncation,  improvement beyond the first 10 modes is marginal. In particular, the relative $H_\infty$ error drops below $10^{-5}$ after the first 10 poles. \textsc{sadpa} converged to the 10 most dominant poles of $\BH_{ext}$ after 25 iterations using $s_0=-0.1$. The corresponding reduced-order output recorded a maximum pointwise relative error of $\approx.3$ and an average pointwise relative error of $\approx .005$. \\

\section{Conclusions} \label{sec6_concl}

Conventional model reduction techniques for linear time-periodic systems typically rely on access to the Floquet transform.  Computing this transform can be prohibitively expensive for large-scale systems, limiting the practicality of these methods. In this work, we introduce the partial Floquet transform based on the identification of effective invariant subspaces which are extracted efficiently by a variant of the dominant pole algorithm \textsc{dpa}. By constructing a partial Floquet transform using only a subset of the system's most dominant poles, we are able to produce effective reduced-order models without requiring full spectral information.

The primary limitation of our proposed approach lies in its heuristic nature (which likewise is shared with the original \textsc{dpa} approach): there are currently no performance guarantees, and it may be difficult to predict in advance whether a given system is well-suited to dominant pole truncation. Addressing this challenge -- either through the development of diagnostic criteria or by integrating the method with more robust reduction frameworks -- remains an active area  of ongoing research.  

\section*{Acknowledgements:}
 This project was supported by the National Science Foundation (NSF) of the United States under Grant No. 2318880.  The authors express their appreciation to Profs. Mark Embree and Serkan Gugercin of Virginia Tech for insightful comments on the analysis underlying Sections 3 and 4.

\bibliography{MORe_BenderBeattie}

\end{document}